\documentclass[letterpaper,11pt,final]{article}%
\usepackage{titlesec,titletoc}
\usepackage{comment}
\usepackage{amscd}
\usepackage[overload]{textcase}
\usepackage{makeidx}
\usepackage{etex}
\usepackage{amsxtra}
\usepackage{graphicx}
\usepackage[amsmath,hyperref,thmmarks]{ntheorem}
\usepackage{natbib}
\usepackage{amsmath}
\usepackage{amsfonts}
\usepackage{amssymb}
\usepackage{natbib}
\usepackage[colorlinks=true,bookmarksnumbered=true,pdfpagemode=None]{hyperref}%
\providecommand{\U}[1]{\protect\rule{.1in}{.1in}}
\theoremnumbering{arabic}
\theoremheaderfont{\scshape}
\RequirePackage{latexsym}
\theorembodyfont{\slshape}
\theoremseparator{.}
\newtheorem{X}{X}[section]

\newtheorem{corollary}[X]{Corollary}

\newtheorem{lemma}[X]{Lemma}
\newtheorem{proposition}[X]{Proposition}

\newtheorem{theorem}[X]{Theorem}

\theorembodyfont{\upshape}
\newtheorem{definition}[X]{Definition}
\newtheorem{example}[X]{Example}

\newtheorem{plain}[X]{}

\newtheorem{remark}[X]{Remark}
\newtheorem{aside}[X]{Aside}
\theorembodyfont{\small}
\theorembodyfont{\normalsize}
\theoremstyle{nonumberplain}
\theoremsymbol{\ensuremath{_\Box}}
\newtheorem{proof}{Proof}

\qedsymbol{\ensuremath{_\Box}}

\contentsmargin{2.0em}
\dottedcontents{section}[2.3em]{}{1.8em}{1pc}
\dottedcontents{subsection}[6.0em]{}{1.8em}{1pc}
\titleformat*{\subsection}{\large\scshape}
\titleformat*{\subsubsection}{\slshape}

\usepackage{titlesec}                                                                            
\titleformat*{\section}{\LARGE\bfseries}
\titleformat*{\subsection}{\Large\itshape}
\titleformat*{\subsubsection}{\scshape}
\titleformat*{\paragraph}{\itshape}

\usepackage{enumitem}
\setitemize[1]{label=$\diamond$\hspace{0.07in}}
\setlist{topsep=0.1em,itemsep=0.1em,parsep=0.1em}

\usepackage[nottoc]{tocbibind}

\usepackage{mathptmx}
\usepackage[scaled=0.92]{helvet}

\usepackage{array}

\usepackage{microtype}

\usepackage{tikz}
\usetikzlibrary{matrix,arrows,positioning,decorations.pathmorphing}
\usepackage{tikz-cd}
 
\usepackage{wrapfig}

\usepackage{etex}

\usepackage{url}
\usepackage{verbatim}

\usepackage[notcite,notref]{showkeys}

\usepackage{mathtools}

\usepackage[overload]{textcase}

\newcommand{\eb}[1]{{\itshape\bfseries#1}}
\renewcommand{\emph}{\eb}

\usepackage{natbib}
\let\cite\citealt

\hypersetup{linkcolor=blue,anchorcolor=blue,citecolor=blue,filecolor=blue,urlcolor=blue}
\hypersetup{pdftitle={Numerical and homological equivalence over finite fields},pdfauthor={J.S. Milne}}
\hypersetup{pdfsubject={standard conjectures},pdfkeywords={Lefschetz standard conjecture, Tate conjecture}}

\usepackage[papersize={6.5in,10in},margin=0.5in]{geometry}

\newcommand{\bcomment}{}
\newcommand{\bfootnotesize}{\begin{footnotesize}}\newcommand\efootnotesize{\end{footnotesize}}
\newcommand{\bquote}{\begin{quote}}\newcommand\equote{\end{quote}}
\newcommand{\bsmall}{\begin{small}}\newcommand\esmall{\end{small}}
\newcommand{\btable}{\begin{table}}\newcommand{\etable}{\end{table}}

\newcommand{\edocument}{\end{document}}

\newcommand{\tableoc}{\tableofcontents}

\def\1{{1\mkern-7mu1}}

\DeclareMathOperator{\End}{End}
\DeclareMathOperator{\ev}{ev}

\DeclareMathOperator{\Gal}{Gal}
\DeclareMathOperator{\GL}{GL}

\DeclareMathOperator{\Hom}{Hom}
\DeclareMathOperator{\id}{id}

\DeclareMathOperator{\ord}{ord}

\DeclareMathOperator{\rank}{rank}

\DeclareMathOperator{\Tr}{Tr}

\newcommand{\Mot}{\mathsf{Mot}}

\newcommand{\Vc}{\mathsf{Vec}}

\begin{document}

\title{On the Tate and Standard Conjectures over Finite Fields (version 1)}
\author{J. S. Milne}
\date{July 4, 2019}
\maketitle

\begin{abstract}
For an abelian variety over a finite field, Clozel (1999) showed that
$l$-homological equivalence coincides with numerical equivalence for
infinitely many $l$, and the author (1999)\nocite{milne1999} gave a criterion
for the Tate conjecture to follow from Tate's theorem on divisors. We
generalize both statements to motives, and apply them to other varieties 
including $K3$ surfaces.
\end{abstract}

\tableoc

\bigskip Let $X$ be a motive for rational equivalence over $\mathbb{F}_{q}$,
and let $\omega$ be the fibre functor defined by a standard Weil cohomology
theory. Assume that the Frobenius element $\pi_{X}$ of $X$ acts semisimply on
$\omega(X)$. Then $\mathbb{Q}{}[\pi_{X}]$ has a unique positive involution
$\alpha\mapsto\alpha^{\prime}$, and we let $S$ denote the algebraic group over
$\mathbb{Q}{}$ such that $S(\mathbb{Q}{})=\{a\in\mathbb{Q}{}[\pi]\mid a\cdot
a^{\prime}=1\}$.

We let $\bigotimes V$ denote the tensor algebra $\bigoplus_{n\in\mathbb{N}{}%
}V^{\otimes n}$ of a finite-dimensional vector space $V$ and $\langle
X\rangle^{\otimes}$ the pseudo-abelian rigid tensor subcategory generated by a
motive $X$ and the Tate object.

\begin{theorem}
\label{xx01}Let $X$ be a motive of weight $m$ over $\mathbb{F}{}_{q}$ whose
Frobenius element $\pi_{X}$ acts semisimply on $\omega(X)$.

\begin{enumerate}
\item The $\mathbb{Q}{}$-algebra\textrm{ }$(\bigotimes\omega(X))^{S}$ is
generated by $\omega(X^{\otimes2})^{S}$:%
\[
(\bigotimes\omega(X))^{S}=\mathbb{Q}{}[\omega(X^{\otimes2})^{S}].
\]

\item If $X^{\otimes2}(m)$ satisfies the Tate conjecture and $S$ is generated
by $\pi_{X}/q^{m/2}$, then the Tate conjecture holds for all motives in
$\langle X\rangle^{\otimes}$.
\end{enumerate}
\end{theorem}

In fact, we prove a more general statement (Theorem \ref{xx24}).

Our next theorem generalizes the theorem of \cite{clozel1999} (see also
\cite{deligne2009}). Let $X$ be a motive over $\mathbb{F}{}_{q}$ for rational
equivalence and $X_{l}$ (resp.\ $X_{\mathrm{num}}$) the image of $X$ in the
category of motives for $l$-adic homological equivalence (resp.\ for numerical equivalence).

\begin{theorem}
\label{xx02} Let $X$ be a motive of weight $m$ over $\mathbb{F}{}_{q}$. If
$X^{\otimes2}(m)$ satisfies the full Tate conjecture, then there exists an
infinite set of prime numbers $P(X)$, depending only on the centre of
$\End(X_{\mathrm{num}})$, such that the functor%
\[
\langle X_{l}\rangle^{\otimes}\rightarrow\langle X_{\mathrm{num}}%
\rangle^{\otimes}%
\]
is fully faithful for all $l\in P(X)$.
\end{theorem}

In fact, we prove a stronger statement (Theorem \ref{xx27} et seq.).

On applying these statements to $h^{1}$ of an abelian variety, we recover the
theorems of Clozel and the author mentioned above. On applying Theorem
\ref{xx01} to the motive of a $K3$ surface $V$ of height $1$, we find that the
Tate conjecture holds for all powers of $V$, thereby recovering a theorem of
Zarhin. On applying Theorem \ref{xx02} to the motive of a general $K3$ surface
$V$, we find that their exist infinitely many primes $l$ such that
$l$-homological equivalence coincides with numerical equivalence for all
powers of $V$.

\begin{aside}
\label{xx03}Our proof of Theorem \ref{xx01} follows the proof for abelian
varieties in \cite{milne1999}. However, our proof of Theorem \ref{xx02} is
simpler than the proofs for abelian varieties in \cite{clozel1999} and
\cite{deligne2009}; in particular, it avoids using the standard conjecture of
Lefschetz type.
\end{aside}

\begin{aside}
\label{xx04}There is the following general procedure for producing infinite
sequences of new theorems. Take a theorem about abelian varieties and restate
and prove it for motives. Now apply the motivic result to the motive of
\textit{any }variety to obtain a new theorem. In practice, this process does
not work well because much more is known to be true about abelian varieties
and their motives than about general motives. Nevertheless, in this article,
we show that it can yield new results.
\end{aside}

\subsection{Notation and Conventions}

All algebraic varieties are smooth and projective. Complex conjugation on
$\mathbb{C}{}$ and its subfields is denoted by $\iota$. The field of $q$
elements, $q$ a power of the prime $p$, is denoted by $\mathbb{F}{}_{q}$. The
symbol $l$ always denotes a prime number, possibly $p$. For a perfect field
$k$, we let $W(k)$ denote the ring of Witt vectors of $k$ and $B(k)$ the field
of fractions of $W(k)$. For tensor categories, we follow the conventions of
\cite{deligneM1982}.

\section{Statement of the Tate conjecture for motives}

\subsection{Categories of motives}

We refer the reader to \cite{scholl1994} for the basic formalism of motives.
In particular, a motive over a field $k$ is a triple $(V,e,m)$ with $V$ a
variety over $k$, $e$ an algebraic correspondence such that $e^{2}=e$, and $m$
an integer. When the K\"{u}nneth components $p_{i}$ of the diagonal are
algebraic, we use them to modify the commutativity constraint. We write
$h^{i}(V)(m)$ for $(V,p_{i},m)$. The category of motives over a field $k$
defined by an admissible equivalence relation $\sim$ is denoted by
$\Mot_{\sim}(k)$. It is a pseudo-abelian rigid tensor category with
$\End(\1)=\mathbb{Q}{}$, and if $V$ is of pure dimension $d$, then the motive
dual to $(V,e,m)$ is%
\[
(V,e,m)^{\vee}=(V,e^{t},d-m)
\]
(\cite{saavedra1972}, VI, 4.1.3.5). We let $\gamma$ denote the functor%
\[
X\rightsquigarrow\Hom(\1,X)\colon\Mot_{\sim}(k)\rightarrow\Vc_{\mathbb{Q}{}}.
\]
For example, if $X=(V,e,m)$, then $\gamma(X)=eA_{\mathrm{\sim}}^{m}(V)$
(algebraic cycles of codimension $m$ modulo $\sim$ with $\mathbb{Q}{}%
$-coefficients). Similarly, $\gamma(X^{\vee})=e^{t}A_{\mathrm{\sim}}^{d-m}%
(V)$, and the intersection product $A_{\mathrm{\sim}}^{m}(V)\times
A_{\mathrm{\sim}}^{d-m}(V)\rightarrow\mathbb{Q}$ defines an intersection
product
\[
\gamma(X)\times\gamma(X^{\vee})\rightarrow\mathbb{Q}{}.
\]

We write $\mathrm{rat}$ for rational equivalence, $\mathrm{num}$ for numerical
equivalence, and $\mathrm{hom}(l)$ (or just $l$) for the equivalence relation
defined by the standard $l$-adic Weil cohomology theory, We have categories of
motives and tensor functors%
\[
X\rightsquigarrow X_{l}\rightsquigarrow X_{\mathrm{num}}\colon
\Mot_{\mathrm{rat}}(\mathbb{F}{}_{q})\rightarrow\Mot_{\mathrm{hom}%
(l)}(\mathbb{F}{}_{q})\rightarrow\Mot_{\mathrm{num}}(\mathbb{F}{}_{q}).
\]
The category $\Mot_{\mathrm{num}}(\mathbb{F}{}_{q})$ is tannakian (Jamsen 1992).

\subsection{The Tate conjectures}

Let $X$ be a motive in $\Mot_{\mathrm{rat}}(\mathbb{F}{}_{q})$ and $l$ a prime
number. The standard $l$-adic Weil cohomology theory $V\rightsquigarrow
H_{l}^{\ast}(V)$ is crystalline cohomology if $l=p$ and \'{e}tale cohomology
if $l\neq p$. Let $Q$ denote its coefficient field; thus $Q=B(\mathbb{F}{}%
_{q})$ if $l=p$ and $\mathbb{Q}{}=\mathbb{Q}{}_{l}$ if $l\neq p$. Let
$\omega_{l}$ denote the tensor functor $\Mot_{\mathrm{rat}}(\mathbb{F}{}%
_{q})\rightarrow\Vc_{Q}$ defined by $H_{l}^{\ast}$. The vector space
$\omega_{l}(X)$ is equipped with a $Q$-linear Frobenius operator $\pi=\pi_{X}%
$. For example, if $l=p$, then $\omega_{l}(X)$ is an $F$-isocrystal, and
$\pi_{X}=F^{n}$, where $q=p^{n}$.

The functor $\omega_{l}$ defines a $\mathbb{Q}{}$-linear map%
\begin{equation}
\gamma(X)\overset{\textup{{\tiny def}}}{=}\Hom(\1,X)\overset{\omega
_{l}}{\longrightarrow}\Hom(\omega_{l}(\1),\omega_{l}(X))\simeq\omega_{l}(X),
\label{e11}%
\end{equation}
and we let $A_{l}(X)$ denote its image. Thus $A_{l}(X)$ is a $\mathbb{Q}{}%
$-subspace of $\omega_{l}(X)$ canonically isomorphic to $\gamma(X_{l})$. The
inclusion $A_{l}(X)\subset\omega_{l}(X)$ extends to a $Q$-linear map%
\begin{equation}
c\otimes a\mapsto ca\colon Q\otimes_{\mathbb{Q}{}}A_{l}(X)\rightarrow
\omega_{l}(X)^{\pi_{X}}. \label{e10}%
\end{equation}

Let $(X^{\vee},\ev)$ be the dual motive. The morphism $\ev\colon X^{\vee
}\otimes X\rightarrow\1$ induces a perfect pairing $\omega_{l}(X^{\vee}%
)\times\omega_{l}(X)\rightarrow Q$, compatible with the Frobenius maps, which
restricts to pairings%

\begin{equation}
\begin{tikzpicture}[scale=0.8,baseline=(current bounding box.center), text height=1.5ex, text depth=0.25ex] \node (a) at (0,0) {$\omega(X^{\vee})^{\pi}$}; \node (t) at (1.2,0) {$\times$}; \node (b) at (2.2,0) {$\omega(X)^{\pi}$}; \node (c) at (4.3,0) {$Q$}; \node (d) [below=0.7cm of a] {$A_l(X^{\vee})$}; \node [below=0.7cm of t] {$\times$}; \node (e) [below=0.7cm of b] {$A_l(X)$}; \node (f) [below=0.7cm of c] {$\mathbb{Q}$\rlap{.}}; \node (ad) [below=0.1cm of a] {\rotatebox{90}{$\subset$}}; \node (be) [below=0.1cm of b] {\rotatebox{90}{$\subset$}}; \node (cd) [below=0.1cm of c] {\rotatebox{90}{$\subset$}}; \draw[->,font=\scriptsize,>=angle 90] (b) edge (c) (e) edge (f); \end{tikzpicture} \label{e13}%
\end{equation}
The lower pairing is intersection product.

There are the following conjectural statements.

\begin{description}
[font=\normalfont]

\item[$T(X,l)$:] The map (\ref{e10}) is surjective, i.e., $Q{}\cdot
A_{l}(X)=\omega_{l}(X)^{\pi_{X}}$.

\item[$I(X,l)$:] The map $(\ref{e10})$ is injective, i.e., $Q{}\otimes
A_{l}(X)\hookrightarrow\omega_{l}(X)$.

\item[$D(X,l)$:] The right kernel $N(X)$ of the pairing $A_{l}(X^{\vee})\times
A_{l}(X)\rightarrow\mathbb{Q}{}$ is zero; equivalently, the map $\gamma
(X_{l})\rightarrow\gamma(X_{\mathrm{num}})$ is an isomorphism.

\item[$S(X,l)$:] The eigenvalue $1$ of $\pi_{X}$ on $\omega_{l}(X)$ is
semisimple (i.e., not a multiple root of the minimum polynomial of $\pi_{X}).$
\end{description}

\begin{remark}
\label{xx05}(a) When $X=h^{2j}(V)(j)$, these become the statements in
\cite{tate1994}.

(b) If $X$ is of weight $m\neq0$, then these statements are trivially true
because $\gamma(X_{l})=0=\gamma(X_{l}^{\vee})$ and $1$ is not an eigenvalue of
$\pi_{X}$ on $\omega_{l}(X)$.

(c) Let $X$ have weight $2m$. Once we have chosen an isomorphism $Q\rightarrow
Q(1)$, we can identify (\ref{e10}) with a map
\[
Q\otimes_{\mathbb{Q}{}}A_{l}(X(m))\rightarrow\omega_{l}(X(m))\simeq\omega
_{l}(X).
\]
Then $T(X(m),l)$ becomes the statement that $Q{}\cdot A_{l}(X)=\omega
_{l}(X)^{\pi_{X}/q^{m}}$.
\end{remark}

\begin{theorem}
\label{xx06}Let $X$ be a motive of weight $2m$ over $\mathbb{F}{}_{q}$, and
let $l$ be a prime number (possibly $p$). The following statements are equivalent:

\begin{enumerate}
\item $\dim_{\mathbb{Q}{}}A_{l}(X(m))/N=\dim_{Q}\omega_{l}(X)^{\pi/q^{m}}$;

\item $T(X(m),l)+D(X(m),l)$;

\item $T(X(m),l)+T(X^{\vee}(-m),l)+S(X(m),l);$

\item $T(X(m),l)+T(X^{\vee}(-m),l)+D(X(m),l)+D(X^{\vee}(-m),l)+I(X(m),l)\qquad
\qquad\qquad\linebreak\hspace*{5.6cm}+I(X^{\vee}(-m),l)+S(X(m),l)+S(X^{\vee
}(-m),l);$

\item the order of the pole of the zeta function $Z(X,t)$ at $t=q^{-m}$ is
equal to the dimension of the $\mathbb{Q}$-vector space $\gamma
(X(m)_{\mathrm{num}})$.
\end{enumerate}
\end{theorem}

\begin{proof}
The hard Lefschetz theorem for $l$-adic cohomology shows that $\omega_{l}(X)$
and $\omega_{l}(X^{\vee})$ are isomorphic together with their Frobenius
operators, and so
\[
\omega_{l}(X)^{\pi_{X}/q^{m}}\approx\omega_{l}(X^{\vee})^{\pi_{X^{\vee}}%
/q^{m}}.
\]
In particular, they have the same dimension. We can now apply the arguments in
Tate 1994, \S 2, to the diagram%
\[
\begin{tikzcd}[column sep=small]
Q\otimes A_l(X(m)) \arrow{r}{b}\arrow{d}{a}
&\omega_l(X(m))^{\pi/{q^m}}\arrow{r}{c}
&(\omega_l(X(-m)^{\vee})^{\pi/{q^m}})^{\ast}\arrow{d}{d}\\
Q\otimes(A_l(X(m))/N)\arrow[hookrightarrow]{r}{f}
&Q\otimes\Hom_{\mathbb{Q}}(A_l(X(-m)^{\vee}),\mathbb{Q}) \arrow[hookrightarrow]{r}{e}%
&(Q\otimes A_l(X(-m)^{\vee}))^{\ast}.
\end{tikzcd}
\]
Here $(-)^{\ast}=\Hom_{Q}(-,Q)$, $b$ is the map (\ref{e10}), and $d$ is the
linear dual of the similar map for $X^{\vee}$. The maps $e$ and $f$ are
defined by the pairings in (\ref{e13}), and the maps $a$ and $e$ are obvious.
\end{proof}

Statement (e) of the theorem is independent of $l$, and so%
\[
T(X(m),l)+D(X(m),l)\text{ for one }l\iff T(X(m),l)+D(X(m),l)\text{ for all
}l\text{.}%
\]
We call $T(X(m),l)$ the \textit{Tate conjecture} (for $X(m),l)$, and
equivalent statements of the theorem the \textit{full Tate conjecture} (for
$X(m))$.

\begin{remark}
\label{xx07}All four conjectures are stable under passage to direct sums,
direct summands, and duals (but not tensor products or Tate twists).

For example, we show that $T(X,l)$ implies $T(X^{\vee},l)$. Because $T$ is
stable under passage to direct sums and direct summands, it suffices to do
this for $X=h^{2i}(V)(i)$. From the hard Lefschetz theorem for the $l$-adic
Weil cohomology, we get a diagram%
\[
\begin{tikzcd}
Q\otimes A_{l}(X)\arrow{r}\arrow{d}{\textrm{onto}}
&Q\otimes A_{l}(X^{\vee})\arrow{d}\\
H_{l}^{2i}(V)(i)^{\pi}\arrow{r}{\simeq}&H_{l}^{2d-2i}(V)(d-i)^{\pi},
\end{tikzcd}
\]
from which the statement follows.

Let $X$ be a motive; if one of the conjectures holds for all motives
$X^{\otimes r}(m)$, $r\in\mathbb{N}{}$, $m\in\mathbb{Z}{}$, of weight $0$,
then it holds for all objects of $\langle X\rangle^{\otimes}$.
\end{remark}

\subsection{Characteristic polynomials}

Let $(\mathsf{C},\otimes)$ be a pseudo-abelian rigid tensor category with
$\End(\1)=\mathbb{Q}$. Every object $X$ of $\mathsf{C}$ admits a dual
$(X^{\vee},\ev)$, and we have maps%
\[
\underline{\Hom}(X,X)\simeq X^{\vee}\otimes X\overset{\ev}{\longrightarrow
}\1.
\]
On applying the functor $\Hom(\1,\,\,)$ to the composite of these maps, we
obtain the trace map%
\[
\alpha\mapsto\Tr(\alpha|X)\colon\End(X)\rightarrow\End(\1)=\mathbb{Q}{}.
\]
The $\rank$ of $X$ is defined to be $\Tr(\id_{X}|X)$. These constructions
commute with tensor functors. If $(\mathsf{C},\otimes)$ is tannakian, then,
for every fibre functor $\omega\colon\mathsf{C}\rightarrow\Vc_{Q}$ ($Q$ a
field containing $\mathbb{Q}{}$), we have
\begin{align*}
\Tr(\alpha|X)  &  =\Tr(\omega(\alpha)\mid\omega(X))\\
\rank(X)  &  =\dim_{Q}(\omega(X))\in\mathbb{N}{}.
\end{align*}

From now on, we assume that $(\mathsf{C},\otimes)$ admits a tensor functor to
a tannakian category with $\End(\1)=\mathbb{Q}{}$. For an object $X$ and
integer $r\geq0{}$, there is a morphism%
\[
a(r)=\sum_{\sigma\in S_{r}}\mathrm{sgn}(\sigma)\cdot\sigma\colon X^{\otimes
r}\rightarrow X^{\otimes r}\text{,}\quad S_{r}=\text{symmetric group.}%
\]
Then $a(r)/r!$ is an idempotent in $\End(X^{\otimes r})$, and we define
$\bigwedge\nolimits^{r}X$ to be its image. The characteristic polynomial of an
endomorphism $\alpha$ of $X$ is
\[
P_{\alpha}(X,t)=c_{0}+c_{1}t+\cdots+c_{r-1}t^{r-1}+t^{r},\quad r=\rank(X),
\]
where%
\[
c_{r-i}=(-1)^{i}\Tr(\alpha\mid\bigwedge\nolimits^{i}X)=(-1)^{i}\Tr(\frac
{a(i)}{i!}\circ\alpha^{\otimes i}\mid X^{\otimes i}).
\]
This definition commutes with tensor functors.

\section{Weil numbers and tori}

\subsection{Invariants of torus actions}

Let $Q$ be a field of characteristic zero with algebraic closure
$Q^{\mathrm{al}}$, and let $S$ be a group of multiplicative type over $Q$
acting on a $Q$-vector space $V$. For a character $\chi$ of $S$, let $V_{\chi
}$ denote the subspace of $V_{Q^{\mathrm{al}}}$ on which $S$ acts through
$\chi$. Then
\[
V_{Q^{\mathrm{al}}}=\bigoplus_{\chi\in X^{\ast}(S)}V_{\chi},\quad X^{\ast
}(S)\overset{\textup{{\tiny def}}}{=}\Hom_{Q^{\mathrm{al}}}(S,\mathbb{G}%
_{m}),
\]
and the $\chi$ for which $V_{\chi}\neq0$ are called the weights of $S$ in $V$.

\begin{lemma}
\label{xx08}If the weights of $S$ in $V$ can be numbered $\xi_{1},\ldots
,\xi_{2m},\ldots,\xi_{2m+r}$ in such a way that the $\mathbb{Z}{}$-module of
relations among the $\xi_{i}$ is generated by the relations%
\begin{align*}
\xi_{i}+\xi_{m+i}  &  =0,\quad i=1,\ldots,m,\\
2\xi_{2m+i}  &  =0,\quad i=1,\ldots,r,
\end{align*}
then $\left(  \bigotimes V\right)  ^{S}$ is generated as a $Q$-algebra by
$(V^{\otimes2})^{S}$.
\end{lemma}

\begin{proof}
Clearly $Q[(V^{\otimes2})^{S}]\subset(\bigotimes V)^{S}$. It suffices to prove
that the two become equal after tensoring with $Q^{\mathrm{al}}$ (because
forming invariants commutes with passing to a field exension). Therefore, we
may suppose that $Q$ is algebraically closed. Fix an integer $n\geq1$. For a
$2m+r$-tuple $\Sigma=(n_{1},\ldots,n_{2m},\ldots,n_{2m+r})$ of nonnegative
integers with sum $n$, let $[\Sigma]$ denote the character $\sum_{i=1}%
^{2m+r}n_{i}\xi_{i}$ of $S$, and let $V(\Sigma)=\bigotimes_{i=1}^{2m+r}%
V_{\xi_{i}}^{\otimes n_{i}}$. Then $V^{\otimes n}=\bigoplus_{\Sigma}V(\Sigma
)$, and $S$ acts on $V(\Sigma)$ through the character $[\Sigma]$. Therefore,
$(V^{\otimes n})^{S}=\bigoplus_{[\Sigma]=0}V(\Sigma)$. By assumption, a
character $[\Sigma]$ is zero if and only if%
\[%
\begin{array}
[c]{ll}%
n_{i}=n_{m+i}, & i=1,\ldots,m,\\
n_{i}\text{ is even} & i=1,\ldots,r.
\end{array}
\]
and so
\[
\lbrack\Sigma]=0\implies V(\Sigma)=\bigotimes_{i=1}^{m}(V_{\xi_{i}}\otimes
V_{\xi_{m+i}})^{\otimes n_{i}}\otimes\bigotimes_{i=1}^{r}(V_{2m+i}^{\otimes
2})^{\otimes n_{i}/2}\text{.}%
\]
But $V_{\xi_{i}}\otimes V_{\xi_{m+i}}\subset(V^{\otimes2})^{S}$ because
$\xi_{i}+\xi_{m+i}=0$ in $X^{\ast}(T)$, and $V_{2m+i}^{\otimes2}%
\subset(V^{\otimes2})^{S}$ because $2\xi_{2m+i}=0$. It follows that
$(\bigotimes^{n}V)^{S}\subset Q$ $[(V^{\otimes2})^{S}]$.
\end{proof}

An involution $a\mapsto a^{\prime}$ of a $\mathbb{Q}{}$-algebra $F$ is said to
be positive if $\Tr_{F/\mathbb{Q}}(a\,a^{\prime})>0$ for all nonzero $a$ in
$F$. A commutative finite $\mathbb{Q}{}$-algebra $F$ admits a positive
involution if and only if it is a product of totally real fields and CM
fields, in which case, the involution acts as the identity map on each real
factor and as complex conjugation on each CM factor; in particular, it is unique.

\begin{proposition}
\label{xx09}Let $(F,\,^{\prime})$ be a finite commutative $\mathbb{Q}{}%
$-algebra with positive involution, and let $S$ denote the algebraic group (of
multiplicative type) such that%
\[
S(R)=\{a\in(F\otimes R)^{\times}\mid a\cdot a^{\prime}=1\},\quad R\text{ a
}\mathbb{Q}{}\text{-algebra.}%
\]
Let $Q$ be a field containing $\mathbb{Q}{}$, and let $V$ be a free
$F\otimes_{\mathbb{Q}{}}Q$-module of finite rank. Then $(\bigotimes V)^{S}$ is
generated as a $Q$-algebra by $(V^{\otimes2})^{S}$:%
\[
(\bigotimes V)^{S}=Q[(V^{\otimes2})^{S}].
\]

\end{proposition}

\begin{proof}
If $F$ is a CM field, then $S$ is a torus. Every embedding $\sigma\colon
F\hookrightarrow Q^{\mathrm{al}}$ defines a character $\xi_{\sigma}$ of $S$,
and the character group of $S$ is the quotient of $\bigoplus\mathbb{Z}{}%
\xi_{\sigma}$ by the subgroup generated by the elements $\xi_{\sigma}%
+\xi_{\iota\circ\sigma}$ (recall that $\iota$ denotes complex conjugation). As
$V$ is a free $F\otimes Q$-module, the weights of $S$ on $V\otimes
_{Q}Q^{\mathrm{al}}$ are precisely the characters $\xi_{\sigma}$, $\sigma
\in\Hom(F,Q^{\mathrm{al}})$, and each has the same multiplicity. Choose a CM
type $\{\varphi_{1},\ldots,\varphi_{m}\}$ for $F$. Then the character group of
$S$ has generators $\{\xi_{\varphi_{1}},\ldots,\xi_{\varphi_{m}},\xi
_{\iota\circ\varphi_{1}},\ldots,\xi_{\iota\circ\varphi_{m}}\}$ and defining
relations%
\[
\xi_{\varphi_{i}}+\xi_{\iota\circ\varphi_{i}}=0,\quad i=1,\ldots,m.
\]
Thus the statement follows from the lemma.

If $F$ is a totally real field, then $S=\mu_{2}$, and the element $-1$ of
$\mu_{2}(\mathbb{Q}{})$ acts on $V$ as multiplication by $-1$. Therefore
$V_{\mathbb{Q}{}^{\mathrm{al}}}=V_{\xi}$ where $\xi$ is the nonzero element of
$X^{\ast}(S)$. Therefore $X^{\ast}(S)$ has generator $\xi$ and defining
relation $2\xi=0$. Thus the statement follows from the lemma.

In the general case, $F$ is a product of CM fields and totally real fields,
and the same arguments as in the last two paragraphs apply.
\end{proof}

\begin{remark}
\label{xx10}Let $(F,^{\prime})$ be as in \ref{xx09}. It follows from the
definition of $S$, that the involution $a\mapsto a^{\prime}$ acts on the
algebraic group $S$ by sending each element to its inverse. Therefore, it acts
on $X^{\ast}(S)$ as $-1$.
\end{remark}

\subsection{Application to the tori attached to Weil numbers}

By an algebraic number, we mean an element of a field algebraic over
$\mathbb{Q}{}$. We let $\mathbb{Q}{}^{\mathrm{al}}$ denote the algebraic
closure of $\mathbb{Q}{}$ in $\mathbb{C}{}$.

\begin{definition}
\label{xx11}An algebraic number $\pi$ is said to be a Weil\emph{ }$q$-number
of weight $m$ if

\begin{enumerate}
\item for every embedding $\rho\colon\mathbb{Q}{}[\pi]\hookrightarrow
\mathbb{Q}^{\mathrm{al}}$, $|\rho(\pi)|=q^{m/2};$

\item for some $n$, $q^{n}\pi$ is an algebraic integer.
\end{enumerate}
\end{definition}

Let $\pi$ be Weil $q$-number of weight $m$. Condition (a) implies that
$\pi\mapsto\pi^{\prime}=q^{m}/\pi$ defines an involution $\alpha\mapsto
\alpha^{\prime}$ of $\mathbb{Q}{}[\pi]$ such that $\rho(\alpha^{\prime}%
)=\iota(\rho(\alpha))$ for all embeddings $\rho\colon\mathbb{Q}{}%
[\pi]\hookrightarrow\mathbb{Q}{}^{\mathrm{al}}$. Therefore $\mathbb{Q}{}[\pi]$
is a CM-field or a totally real field (according as $\pi^{\prime}\neq\pi$ or
$\pi^{\prime}=\pi$).

\begin{plain}
\label{xx12}Let $\mathbb{Q}{}[\pi]$ be a finite $\mathbb{Q}{}$-algebra that
can be written as a product of fields
\[
\mathbb{Q}{}[\pi]=\prod\nolimits_{i}\mathbb{Q}{}[\pi_{i}],\quad\pi
\leftrightarrow(\pi_{1},\ldots,\pi_{r}),
\]
with each $\pi_{i}$ a Weil $q$-number of weight $m$. Let $^{\prime}$ be the
(unique) positive involution of $\mathbb{Q}{}[\pi]$ and $S$ the algebraic
group attached to $(\mathbb{Q}{}[\pi],\,^{\prime})$ as in \ref{xx09}. Note
that $\pi/q^{m/2}\in S(\mathbb{Q}{}^{\mathrm{al}})$. We say that $\pi/q^{m/2}$
\emph{generates} $S$ if there does not exist a proper algebraic subgroup $H$
of $S$ such that $\pi/q^{m/2}\in H(\mathbb{Q}{}^{\mathrm{al}})$. Note that, if
$S_{i}$ is the group attached to $\mathbb{Q}{}[\pi_{i}]$, then $S=\prod S_{i}%
$, and $\pi/q^{m/2}$ generates $S$ if and only if $\pi_{i}/q^{m/2}$ generated
$S_{i}$ for all $i$.
\end{plain}

\begin{proposition}
\label{xx13}Let $Q$ be a field containing $\mathbb{Q}{}$, and let $V$ be a
free $\mathbb{Q}{}[\pi]\otimes_{\mathbb{Q}{}}Q$-module of finite rank. If
$\pi/q^{m/2}$ generates $S$, then%
\[
(\bigotimes V)^{\pi/q^{m/2}}=Q[(V^{\otimes2})^{S}].
\]

\end{proposition}

\begin{proof}
According to Proposition \ref{xx09}, $(\bigotimes V)^{S}=Q[(V^{\otimes2}%
)^{S}]$, and so it remains to show that $(\bigotimes V)^{\pi/q^{m/2}%
}=(\bigotimes V)^{S}$. As $\pi/q^{m/2}\in S(\mathbb{Q}^{\mathrm{al}}{})$,
certainly $(\bigotimes V)^{\pi/q^{m/2}}\supset(\bigotimes V)^{S}$. Conversely,
let $v\in$ $(\bigotimes V)^{\pi/q^{m/2}}$, and let $H$ be the algebraic
subgroup of $\GL_{V}$ of elements fixing $v$. Then $\pi/q^{m/2}\in
H(\mathbb{Q}^{\mathrm{al}})$, and so $S\subset H$ because $\pi/q^{m/2}$
generates $S$. We conclude that $S$ fixes $v.$
\end{proof}

\begin{aside}
\label{xx14}Let $\mathbb{Q}{}[\pi]$ and $S$ be as in \ref{xx12}, and assume
that $\pi$ has no real conjugates. Let $\pi_{1},\iota\pi_{1},\ldots,\pi
_{g},\iota\pi_{g}$ be the conjugates of $\pi$ in $\mathbb{Q}{}^{\mathrm{al}}$.
The following conditions are equivalent:

\begin{enumerate}
\item $S$ is generated by $\pi/q^{m/2}$;

\item if%
\[
\pi_{1}^{n_{1}}\cdots\pi_{g}^{n_{g}}=q^{n},\quad n_{i}\text{, }n\in
\mathbb{Z}{},
\]
then $n_{1}=\cdots=n_{g}=0=n$;

\item if
\[
\pi_{1}^{n_{1}}\cdot\iota\pi_{1}^{n_{1}^{\prime}}\cdot\cdots\cdot\pi
_{g}^{n_{g}}\cdot\iota\pi_{g}^{n_{g}^{\prime}}=q^{n},\quad n_{i}%
,\,n_{i}^{\prime},\,n\in\mathbb{Z}{},
\]
then $n_{i}=n_{i}^{\prime}$ for all $i$ and $m(n_{1}+\cdots+n_{g})=n$.
\end{enumerate}
\end{aside}

\subsection{Weil numbers $\pi$ such that $\pi/q^{m/2}$ generates $S$.}

Let $\mathbb{Q}{}[\pi]=\prod\mathbb{Q}{}[\pi_{i}]$ and $S$ be as in
\ref{xx12}. We give some examples where $\pi/q^{m/2}$ generates $S$ (following
\cite{kowalski2006}, \cite{lenstraZ1993}, and \cite{zarhin1991}).

We (ambiguously) write $\iota$ for the unique positive involution of
$\mathbb{Q}{}[\pi]$ (so $\rho\circ\iota=\iota\circ\rho$ for all $\rho
\colon\mathbb{Q}{}[\pi]\rightarrow\mathbb{Q}{}^{\mathrm{al}}$). By a $p$-adic
prime of $\mathbb{Q}{}[\pi]$ we mean a prime ideal of the integral closure of
$\mathbb{Z}{}$ in $\mathbb{Q}{}[\pi]$ lying over $p$. Thus $\mathbb{Q}{}%
_{p}\otimes_{\mathbb{Q}{}}\mathbb{Q}{}[\pi]\simeq\prod\mathbb{Q}{}[\pi]_{v}$
where $v$ runs over the $p$-adic primes of $\mathbb{Q}{}[\pi]$ and each field
$\mathbb{Q}{}[\pi]_{v}$ is the completion of some $\mathbb{Q}{}[\pi_{i}]$ at a
$p$-adic prime of it. The degree of $v$ is $[\mathbb{Q}{}[\pi]_{v}%
\colon\mathbb{Q}{}_{p}]$.

\begin{example}
\label{xx15}Assume that there exists a $p$-adic prime $v_{1}$ of $\mathbb{Q}%
{}[\pi]$ of degree $1$ and such that
\[
\frac{\ord_{v}(\pi)}{\ord_{v}(q^{m})}=\left\{
\begin{array}
[c]{ll}%
0 & \text{if }v=v_{1}\\
1/2 & \text{if }v\neq v_{1},\iota v_{1}\\
1 & \text{if }v=\iota v_{1}.
\end{array}
\right.
\]
For a homomorphism $\rho\colon\mathbb{Q}{}[\pi]\rightarrow\mathbb{Q}%
{}^{\mathrm{al}}$ and a $p$-adic prime $w$ of $\mathbb{Q}{}^{\mathrm{al}}$, we
have
\[
\frac{\ord_{w}(\rho(\pi/q^{m/2}))}{\ord_{w}(q^{m})}=\left\{
\begin{array}
[c]{rl}%
-1/2 & \text{if }\rho^{-1}(w)=v_{1}\\
0 & \text{if }\rho^{-1}(w)\neq v_{1},\iota v_{1}\\
1/2 & \text{if }\rho^{-1}(w)=\iota v_{1}\text{.}%
\end{array}
\right.
\]
Let $\chi=\sum n(\rho)\rho$ be a character of $(\mathbb{G}_{m})_{\mathbb{Q}%
{}[\pi]/\mathbb{Q}{}}$. For a $p$-adic prime $w$ of $\mathbb{Q}{}%
^{\mathrm{al}}$,%
\[
\ord_{w}(\chi(\pi/q^{m/2}))=\frac{\ord_{v_{1}}(q^{m})}{2}(-n(\rho
)+n(\iota\circ\rho)),
\]
where $\rho$ is the unique embedding of $\mathbb{Q}{}[\pi]$ into $\mathbb{Q}%
{}^{\mathrm{al}}$ such that $\rho^{-1}(w)=v_{1}$ (the uniqueness uses that
$v_{1}$ has degree $1$). Therefore, $\chi(\pi/q^{m/2})=1$ if and only if
$\chi=\iota\circ\chi$, i.e., if and only if $\chi$ is trivial on $S$. This
shows that $\pi/q^{m/2}$ generates $S$.
\end{example}

\begin{example}
\label{xx16}Assume that there exists a $p$-adic prime $v_{1}$ of $\mathbb{Q}%
{}[\pi]$ with decomposition group $\{1,\iota\}$ and such that
\[
\frac{\ord_{v}(\pi)}{\ord_{v}(q^{m})}=\left\{
\begin{array}
[c]{ll}%
1/2 & \text{if }v=v_{1}\\
0\text{ or }1 & \text{otherwise.}%
\end{array}
\right.
\]
Then $\pi/q^{m/2}$ generates $S$ (see, for example, Milne 2001, A.9).
\end{example}

Let $\pi$ be a Weil $q$-number of weight $m>0$ with no real conjugates. We say
that $\pi$ is \emph{ordinary} if its characteristic polynomial%
\[
X^{2g}+a_{1}X^{2g-1}+\cdots+a_{g}X^{g}+\cdots+q^{gm}%
\]
has middle coefficient $a_{g}$ a $p$-adic unit. From the theory of Newton
polygons, this means that, for each $p$-adic prime $v$ of $\mathbb{Q}{}[\pi]$,
either $\ord_{v}(\pi)=0$ (hence $\ord_{v}(\iota\pi)=m\ord_{v}(q)$) or
$\ord_{v}(\iota\pi)=0$. An ordinary Weil $q$-number is an algebraic integer.

\begin{example}
\label{xx17}Let $\pi$ be an ordinary Weil $q$-number of weight $m>0$ and
degree $2g$, and let $L$ be the splitting field in $\mathbb{Q}^{\mathrm{al}}%
{}$ of the characteristic polynomial of $\pi$. Assume that, for each complex
conjugate pair $\pi^{\prime}$, $\iota\pi^{\prime}$ of conjugates of $\pi$ in
$L$, there exists an element of $\Gal(L/\mathbb{Q}{})$ fixing $\pi^{\prime}$
and $\iota\pi^{\prime}$ and acting as $\iota$ on the remaining conjugates of
$\pi$ in $L$. This condition holds, for example, if $\Gal(L/\mathbb{Q}{})$ is
the full group of permutations of the set of conjugates of $\pi$ preserving
the subsets $\{\pi^{\prime},\iota\pi^{\prime}\}$.

We claim that $\pi/q^{m/2}$ generates $S$. To prove this, we fix a $p$-adic
prime $v$ of $L$ and we normalize $\ord_{v}$ so that $\ord_{v}(q)=1$. We
number the conjugates of $\pi$ in $L$ so that $\ord_{v}(\pi_{i})=m$ and
$\ord_{v}(\iota\pi_{i})=0$ for $i=1,\ldots,g$. We may suppose that $g>1$ as
the case $g=1$ is trivial.

Suppose that there is a relation%
\begin{equation}
\pi_{1}^{n_{1}}\cdots\pi_{g}^{n_{g}}=q^{n},\quad n_{i}\text{, }n\in
\mathbb{Z}{}, \label{e0}%
\end{equation}
in $L$. Let $\alpha=\pi_{1}^{n_{1}}\cdots\pi_{g}^{n_{g}}$. Then $\alpha
\bar{\alpha}=q^{2n}$, and so%
\[
m(n_{1}+\cdots+n_{g})=2n\text{.}%
\]
Let $\sigma_{i}$ be the element of $\Gal(L/\mathbb{Q}{})$ fixing $\pi_{i}$ and
$\iota\pi_{i}$ and interchanging $\pi_{j}$ and $\iota\pi_{j}$ for $j\neq i$.
On applying $\sigma_{i}$ to (\ref{e0}) and then $\ord_{v}$, we find that%
\[
mn_{i}=n.
\]
As this is true for all $i$, we deduce that $n_{1}=\cdots=n_{g}=0$, i.e.,
(\ref{e0}) is the trivial relation. It follows from this that if%
\[
\pi_{1}^{n_{1}}\cdot\iota\pi_{1}^{n_{1}^{\prime}}\cdot\cdots\cdot\pi
_{g}^{n_{g}}\cdot\iota\pi_{g}^{n_{g}^{\prime}}=q^{n},
\]
then%
\[
n_{i}=n_{i}^{\prime}\text{, for all }i.
\]

Let $\rho_{i}$ be the homomorphism $\mathbb{Q}{}[\pi]\rightarrow L$ sending
$\pi$ to $\pi_{i}$, and let $\chi=\sum_{i=1}^{g}(n_{i}\rho_{i}+n_{i}^{\prime
}(\iota\circ\rho_{i}))$ be a cocharacter of $(\mathbb{G}_{m})_{\mathbb{Q}%
{}[\pi]/\mathbb{Q}{}}$. Then%
\[
\chi(\pi/q^{m/2})=\pi_{1}^{n_{1}}\cdot\iota\pi_{1}^{n_{1}^{\prime}}\cdot
\cdots\cdot\pi_{g}^{n_{g}}\cdot\iota\pi_{g}^{n_{g}^{\prime}}\cdot q^{-n}%
\]
with $n=(n_{1}+n_{1}^{\prime}+\cdots+n_{g}+n_{g}^{\prime})m/2$. Therefore,
\[
\chi(\pi/q^{m/2})=1\iff n_{i}=n_{i}^{\prime}\text{ for all }i\iff\chi
=\iota\circ\chi\iff\chi\text{ is trivial on }S.
\]
This shows that $\pi/q^{m/2}$ generates $S$.
\end{example}

\section{Proof of Theorem 0.1}

We begin with an easy lemma.

\begin{lemma}
\label{xx18}Let $a$ and $b$ be eigenvalues of endomorphisms $\alpha$ and
$\beta$ of vector spaces $V$ and $W$. If the eigenvalue $ab$ of $\alpha
\otimes\beta$ on $V\otimes W$ is semisimple and $b\neq0$, then $a$ is semisimple.
\end{lemma}

\begin{example}
\label{xx19}Let $\alpha$ be an endomorphism of a $k$-vector space $V$. Let $q$
be an element of $k$, and suppose that the eigenvalues of $\alpha$ on $V$
occur in pairs $(a,a^{\prime})$ with $aa^{\prime}=q$. If the eigenvalue $q$ of
$\alpha\otimes\alpha$ on $V\otimes V$ is semisimple, then $\alpha$ acts
semisimply on $V$.
\end{example}

We fix a standard Weil cohomology $H$ on the varieties over $\mathbb{F}_{q}{}%
$, and we let $\Mot_{H}(\mathbb{F}{}_{q})$ denote the corresponding category
of motives. It is a pseudo-abelian rigid tensor category such that
$\End(\1)=\mathbb{Q}{}$ (Saavedra 1972, 4.1.3.5). By a motive, we shall mean
an object of $\Mot_{H}(\mathbb{F}_{q})$. We let $Q$ denote the field of
coefficients of $H$ and $\omega$ the tensor functor $\Mot_{H}(\mathbb{F}{}%
_{q})\rightarrow\Vc_{Q}$ defined by $H$.

The category $\Mot_{\mathrm{num}}(\mathbb{F}_{q}{})$ of numerical motives is a
semisimple tannakian category, and there is a (quotient) tensor functor
$q\colon\Mot_{H}(\mathbb{F}_{q}{})\rightarrow\Mot_{\mathrm{num}}%
(\mathbb{F}_{q}{})$. For a motive $X$, the kernel of $\End(X)\rightarrow
\End(qX)$ is a nilpotent ideal $N{}$ equal to the Jacobson radical of
$\End(X)$ (Jannsen 1992, Thm 1, Cor.\thinspace1). In particular, $\End(X)/N$
is separable over $\mathbb{Q}{}$, and so the $\mathbb{Q}{}$-bilinear pairing%
\begin{equation}
\End(X)\times\End(X)\rightarrow\mathbb{Q}{},\quad\alpha,\beta\mapsto
\Tr(\alpha\circ\beta) \label{e5}%
\end{equation}
has left and right kernels equal to $N$.

\begin{lemma}
\label{xx20} Let $M\subset\End(X)$ be a $\mathbb{Q}{}$-subspace such that
$M\cap N{}=0$. Then the map $Q\otimes M\rightarrow\End(\omega(X))$ is
injective, and its image is a direct summand of $\End(\omega(X))$ (as a
$Q$-vector space with a Frobenius operator).
\end{lemma}

\begin{proof}
As $M\cap N=0$, there exists a $\mathbb{Q}{}$-subspace $E$ of $\End(X)$ such
that $M$ and $E$ are dual under the pairing (\ref{e5}). The orthogonal space
to $Q\cdot E$ is a complementary submodule to $Q\cdot M$ in $\End(\omega
_{H}(X))$.
\end{proof}

Let $X$ be a motive. The characteristic polynomial $P_{\alpha}(X,t)=\det
(1-\alpha t|X)$ of an endomorphism $\alpha$ of $X$ is monic with coefficients
in $\mathbb{Q}{}$, and its degree is equal to $\rank X$. Moreover, $P_{\alpha
}(X,t)$ is equal to the characteristic polynomial $\det(1-\omega(\alpha
)t\mid\omega(X))$ of $\omega(\alpha)$ acting on the $Q{}$-vector space
$\omega(X)$. See \S 1.

\begin{lemma}
\label{xx21}Let $X$ be a motive and $F$ a subfield of $\End(X)$. The $F\otimes
Q$-module $\omega(X)$ is free of rank $\frac{\rank(X)}{[F\colon\mathbb{Q}{}]}$ .
\end{lemma}

\begin{proof}
If $F\otimes Q$ is again a field, then $\omega(X)$ is certainly free.
Otherwise it will decompose into a product of fields $F\otimes Q=\prod
\nolimits_{i}F_{i}$, and correspondingly $\omega(X)\simeq\bigoplus
\nolimits_{i}W_{i}$ with $W_{i}$ an $F_{i}$-vector space of dimension
$m_{i}\geq0$. Our task is to show that the $m_{i}$ are all equal (in fact, to
$\rank(X)/[F\colon\mathbb{Q}{}]$).

Let $\alpha$ be such that $F=\mathbb{Q}{}[\alpha]$, and let $P_{\alpha
}(F/\mathbb{Q}{},t)$ (resp.\ $P_{\alpha}(F_{i}/Q{},t)$) denote the
characteristic polynomial of $\alpha$ in the field extension $F/\mathbb{Q}{}$
(resp.\ $F_{i}/Q{}$). From the decomposition $F\otimes Q=\prod_{i}F_{i}$, we
find that%
\begin{equation}
P_{{}\alpha}(F/\mathbb{Q},t)=\prod\nolimits_{i}P_{\alpha}(F_{i}/Q{},t)
\label{e2}%
\end{equation}
in $Q[t]$. The polynomial $P_{{}\alpha}(F/\mathbb{Q},t)$ is irreducible in
$\mathbb{Q}{}[t]$, and (\ref{e2}) is its factorization into irreducible
polynomials in $Q[t]$.

From the isomorphism of $F\otimes Q$-modules $\omega(X)\simeq\bigoplus
\nolimits_{i}W_{i}$, we find that%
\[
P_{\alpha}(X,t)=\prod\nolimits_{i}P_{\alpha}(F_{i}/Q{},t)^{m_{i}}.
\]

The two equations show that every monic irreducible factor of $P_{\alpha
}(X,t)$ in $\mathbb{Q}{}[t]$ shares a root with $P_{\alpha}(F/\mathbb{Q}{}%
,t)$, and therefore equals it. Hence $P_{\alpha}(X,t)=P_{\alpha}%
(F/\mathbb{Q}{},t)^{m}$ for some integer $m$, and each $m_{i}=m$. On equating
the degrees, we find that $\rank X=m[L\colon\mathbb{Q}{}]$.
\end{proof}

\begin{plain}
\label{xx22}According to Wedderburn's Principal Theorem (\cite{albert1939},
III, Theorem 23), the homomorphism of $\mathbb{Q}{}$-algebras
$\End(X)\rightarrow\End(X)/N$ admits a section, unique up to conjugation.
Choose one, and let $\End(X)^{\prime}$ be its image; thus%
\[
\End(X)=N\oplus\End(X)^{\prime}\text{.}%
\]
Let $\pi$ denote the image of $\pi_{X}$ in $\End(X)^{\prime}$ --- it is
independent of the choice of the section because $\pi_{X}$ lies in the centre
of $\End(X)$. As $\End(X)^{\prime}$ is separable and $\pi$ is contained in its
centre, $\mathbb{Q}{}[\pi]$ is a product of fields. Moreover, as $\pi$ is the
image of $\pi_{X}$ under a homomorphism of $\mathbb{Q}{}$-algebras, it
satisfies $P_{\pi_{X}}(X,t)$. Thus we are in the situation of \ref{xx12} with
$m$ equal to the weight of $X$. Let $S$ be the algebraic group defined by
$\mathbb{Q}{}[\pi]$, as in \ref{xx09}. According to Lemma \ref{xx20}, the map
$Q\otimes\mathbb{Q}{}[\pi]\rightarrow\End(\omega(X))$ is injective.
\end{plain}

\begin{lemma}
\label{xx23}Let $X$ be a motive of weight $m$ in $\Mot_{\hom(l)}(\mathbb{F}%
{}_{q})$. The following conditions are equivalent:

\begin{enumerate}
\item $Q\cdot A_{l}(X^{\otimes2}(m))=\omega(X^{\otimes2})^{\pi/q^{m/2}};$

\item $T(X^{\otimes2}(m),l)+S(X^{\otimes2}(m),l);$

\item $T(X^{\otimes2}(m),l)$ and $\pi_{X}$ acts semisimply on $\omega(X)$.
\end{enumerate}

\noindent Each statement implies

\begin{enumerate}
\setcounter{enumi}{3}

\item $Q\cdot A_{l}(X^{\otimes2}(m))\supset\omega(X^{\otimes2})^{S}$,
\end{enumerate}

\noindent and is equivalent to it if $\pi/q^{m/2}$ generates $S$.
\end{lemma}

\begin{proof}
(a)$\iff$(b\.{)}. We have%
\[
Q\cdot A_{l}(X^{\otimes2}(m))\subset\omega(X^{\otimes2})^{\pi_{X}/q^{m/2}%
}\subset\omega(X^{\otimes2})^{\pi/q^{m/2}}\text{.}%
\]
The first inclusion is an equality if and only if $T(X^{\otimes2}(m),l)$
holds, and the second inclusion is an equality if and only if $1$ is a
semisimple eigenvalue of $\pi_{X}/q^{m/2}$ on $\omega(X)$.

(b)$\iff$(c\.{)}. From \ref{xx19} we see that $\pi_{X}$ acts semisimply on
$\omega(X)$ if and only if $q^{m}$ is a semisimple eigenvalue of $\pi
_{X}\otimes\pi_{X}$ on $\omega(X^{\otimes2})$..

\noindent For the final statement, note that $\omega(X^{\otimes2}%
)^{\pi/q^{m/2}}\supset\omega(X^{\otimes2})^{S}$ because $\pi/q^{m/2}\in
S(\mathbb{Q}{}^{\mathrm{al}})$, and equals it if $\pi/q^{m/2}$ generates $S$.
\end{proof}

\begin{theorem}
\label{xx24}Let $X$ be a motive of weight $m$ in $\Mot_{\mathrm{rat}%
}(\mathbb{F}{}_{q})$, and let $\pi$ and $S$ be as in \ref{xx22}.

\begin{enumerate}
\item The $Q$-algebra $(\bigotimes\omega(X))^{S}$ is generated by
$\omega(X^{\otimes2})^{S}$:%
\[
(\bigotimes\omega(X))^{S}=Q[\omega(X^{\otimes2})^{S}].
\]

\item If $\omega(X^{\otimes2})^{S}$ consists of algebraic classes and $S$ is
generated by $\pi_{X}/q^{m/2}$, then the Tate conjecture holds for $X^{\otimes
n}(\tfrac{mn}{2})$ for all $n\in\mathbb{N}{}$ with $mn$ even.
\end{enumerate}
\end{theorem}

\begin{proof}
(a) Lemma \ref{xx21} allows us to apply Proposition \ref{xx13} to the
$\mathbb{Q}{}$-algebra $\mathbb{Q}{}[\pi]$ acting on the $Q$-vector space
$V=\omega(X)$.

(b) From the hypotheses, we find that%
\[
\left(  \bigotimes\omega(X)\right)  ^{\pi/q^{m/2}}=\left(  \bigotimes
\omega(X)\right)  ^{S}\subset Q[A_{l}(X^{\otimes2}(m))]\subset\bigoplus
\nolimits_{n}Q\cdot A_{l}(X^{\otimes n}(m))
\]
which implies that, for all $n$ with $mn$ even,
\[
\omega(X^{\otimes n})^{\pi/q^{m/2}}\subset Q\cdot A_{l}(X^{\otimes n}%
(\tfrac{mn}{2})).
\]

\end{proof}

We now prove Theorem \ref{xx01}. When $\pi_{X}$ actx semisimply on $\omega
(X)$, statement (a) of Theorem \ref{xx24} becomes statement (a) of Theorem
\ref{xx01}. On the other hand, Lemma \ref{xx23} shows that statement (b) of
Theorem \ref{xx24} implies statement (b) of Theorem \ref{xx01}.

\section{Proof Theorem 0.2}

\begin{lemma}
\label{xx25}Let $S$ be a diagonalizable group acting on a $Q$-vector space
$V$, and let $G$ be the group of $Q$-linear automorphisms of $V$ commuting
with the action of $S$. For every character $\chi$ of $S$, the action of $G$
on $V_{\chi}$ is irreducible.
\end{lemma}

\begin{proof}
Because $S$ is diagonalizable,%
\[
V=\bigoplus_{\chi\in X^{\ast}(S)}V_{\chi},\quad G=\prod_{\chi\in X^{\ast}%
(S)}\GL(V_{\chi}),
\]
from which the statement is obvious.
\end{proof}

Let $H$ be a standard Weil cohomology theory, and let $Q$ be its field of
coefficients. By a motive, we shall mean an object of $\Mot_{H}(\mathbb{F}%
{}_{q})$. For a motive $Y$ of weight $2m$ and a field $Q^{\prime}$ containing
$Q$, we say that an element of $Q^{\prime}\otimes_{Q}\omega(Y)$ is algebraic
if it is in the image of the map%
\[
Q^{\prime}\otimes_{\mathbb{Q}{}}\gamma(Y(m))\rightarrow Q^{\prime}\otimes
_{Q}\omega(Y)
\]
(cf. \ref{xx05}(c)).

Let $X$ be a motive of weight $m$, and let $\pi$ and $S$ be as in \ref{xx22}.
Let $\iota_{\pi}$ denote the unique positive involution of $\mathbb{Q}{}[\pi
]$. Then $\iota_{\pi}$ acts on $X^{\ast}(S)$ as $-1$ (see \ref{xx10}).

\begin{lemma}
\label{xx26}Assume that the elements of $\omega(X^{\otimes2})^{S}$ are
algebraic, and let $n\in\mathbb{N}{}$. If a $Q$-linear map $\alpha\colon
\omega(X^{\otimes n})\rightarrow\omega(X^{\otimes n})$ commutes with the
action of $S$, then it maps algebraic classes to algebraic classes.
\end{lemma}

\begin{proof}
As $(\bigotimes\omega(X))^{S}$ is generated as a $Q$-algebra by $\omega
(X^{\otimes2})^{S}$ (see \ref{xx13}), it consists of algebraic classes. It
follows that the graph of the map $\alpha$ is algebraic, and so it maps
algebraic classes to algebraic classes.
\end{proof}

Let $Y=X^{\vee}$, and let $\bar{Q}$ be a finite Galois extension of $Q$
splitting $\mathbb{Q}{}[\pi]$ (i.e., such that $Q^{\prime}\otimes
_{\mathbb{Q}{}}\mathbb{Q}{}[\pi]$ is a product of copies of $Q^{\prime}$).
From the pairing $\ev\colon Y\otimes X\rightarrow\1$, we get a pairing%
\[
Y^{\otimes n}\times X^{\otimes n}\rightarrow\1,
\]
which in turn gives us a nondegenerate pairing%
\begin{equation}
\omega(Y^{\otimes n})\times\omega(X^{\otimes n})\rightarrow Q. \label{eq4}%
\end{equation}
This is $S$-equivariant, and so its restriction to%
\[
\omega(Y^{\otimes n})_{\chi_{1}}\times\omega(X^{\otimes n})_{\chi_{2}%
}\rightarrow\bar{Q}%
\]
is nondegenerate if $\chi_{1}+\chi_{2}=0$ and zero otherwise.

\begin{theorem}
\label{xx27} Let $\bar{Q}$ be a finite Galois extension of $Q$ splitting
$\mathbb{Q}{}[\pi]$. Assume

\begin{enumerate}
\item there exists an involution $\iota_{Q}$ of $\bar{Q}/Q$ such that
$\rho\circ\iota_{\pi}=\iota_{Q}\circ\rho$ for all homomorphisms $\rho
\colon\mathbb{Q}{}[\pi]\rightarrow\bar{Q}$, and

\item $\omega(X^{\otimes2})^{S}$ consists of algebraic classes.
\end{enumerate}

\noindent Then, for all $n$ with $mn$ even, the pairing%
\begin{equation}
\gamma(Y^{\otimes n}(-\tfrac{mn}{2})\times\gamma(X^{\otimes n}(\tfrac{mn}%
{2}))\rightarrow\mathbb{Q}{} \label{eq5}%
\end{equation}
induced by $Y^{\otimes n}\times X^{\otimes n}\rightarrow\1$ is nondegenerate.
\end{theorem}

\begin{proof}
Fix an $n$, and let $V=\bar{Q}\otimes_{Q}\omega(X^{\otimes n})$ and $W=\bar
{Q}\otimes\omega(Y^{\otimes n})$. Let $G$ be the group of $\bar{Q}$-linear
automorphisms of $V$ commuting with the action of $S$. According to Lemma
\ref{xx26}, the action of $G$ on $\omega(X^{\otimes n})$ preserves the
algebraic classes. Let $\chi$ be a character of $S$ over $\bar{Q}$. As
$V_{\chi}$ is a simple $G$-module (\ref{xx25}), either $V_{\chi}\cap\bar
{Q}\cdot A_{H}(X^{\otimes n}(\tfrac{mn}{2}))$ is the whole of $V_{\chi}$ or it
is zero. If $V\chi$ consists of algebraic classes, so also does $\iota
_{Q}V_{\chi}$ because $\bar{Q}\cdot A_{H}(X^{\otimes n}(\frac{mn}{2}))$ is
stable under the action of $\Gal(\bar{Q}/Q)$. Now%
\[
\iota_{Q}V_{\chi}=V_{\iota_{Q}\circ\chi}=V_{\chi\circ\iota_{\pi}%
}\overset{(\text{\ref{xx10}})}{=}V_{-\chi}\text{.}%
\]
The hard Lefschetz theorem shows that if $V_{-\chi}$ consists of algebraic
classes, then so also does $W_{-\chi}$. It follows that the pairing
\[
\bar{Q}\cdot A_{H}(Y^{\otimes n}(-\tfrac{mn}{2}))\times\bar{Q}\cdot
A_{H}(X^{\otimes n}(\tfrac{mn}{2}))\rightarrow\bar{Q}%
\]
induced by the pairing $W\times V\rightarrow\bar{Q}$ is nondegenerate, and
this implies that the pairing%
\[
A_{H}(Y^{\otimes n}(-\tfrac{mn}{2}))\times A_{H}(X^{\otimes n}(\tfrac{mn}%
{2}))\rightarrow\mathbb{Q}{}%
\]
is nondegenerate. This is isomorphic to the required pairing (\ref{eq5}).
\end{proof}

\begin{corollary}
\label{xx28}Under the hypotheses of the theorem, the functor%
\[
\langle X\rangle^{\otimes}\rightarrow\langle X_{\mathrm{num}}\rangle^{\otimes}%
\]
is an equivalence of categories.
\end{corollary}

\begin{proof}
Immediate consequence of the theorem.
\end{proof}

\begin{corollary}
\label{xx29}Under the hypotheses of the theorem, the maps%
\[
Q\otimes A_{H}(X^{\otimes n}(\tfrac{mn}{2}))\rightarrow\omega(X^{\otimes n})
\]
are injective.
\end{corollary}

\begin{proof}
Let $e_{1},e_{2},\ldots$ be $\mathbb{Q}{}$-linearly independent elements of
$A_{H}(X(\tfrac{m}{2})^{\otimes n})$. According to the theorem, there exist
elements $f_{1},f_{2},\ldots$ of $A_{H}(Y^{\otimes n}(mn/2))$ such that
$\langle e_{i},f_{j}\rangle=\delta_{ij}$. Suppose that $\sum a_{i}\otimes
e_{i}$, $a_{i}\in Q$, maps to zero in $\omega(X^{\otimes n})$. Then
$0=\langle\sum a_{i}e_{i},f_{j}\rangle=a_{j}$.
\end{proof}

\begin{remark}
\label{xx30}Let $X$ be a motive over $\mathbb{F}{}_{q}$ for $l$-adic
homological equivalence.

\begin{enumerate}
\item Let $F$ be the Galois closure of the conjugates of the centre of
$\End(X_{\mathrm{num}})$ in $\mathbb{Q}{}^{\mathrm{al}}$. If complex
conjugation in $\Gal(F/\mathbb{Q}{})$ lies in the decomposition group of
$(l)$, then hypothesis (a) of the theorem holds.

\item If the Tate conjecture holds for $X^{\otimes2}(m)$ and $\pi_{X}$ acts
semisimply on $\omega(X)$, then hypothesis (b) of the theorem holds (see Lemma
\ref{xx23}).
\end{enumerate}
\end{remark}

We now prove Theorem \ref{xx02}. Let $X$ be a motive of weight $m$ in
$\Mot_{\mathrm{rat}}(\mathbb{F}{}_{q})$, and assume the full Tate conjecture
for $X^{\otimes2}(m)$. According to Lemma \ref{xx23}, this implies that
hypotheses (b) of Theorem \ref{xx27} holds for $X_{l}$ for all prime numbers
$l$. Thus, it remains to show that hypothesis (a) holds for an infinite set of
primes, but the Chebotarev density theorem (even the Frobenius density
theorem) implies this (see \ref{xx30}a).

\begin{aside}
\label{xx31}It would be interesting to see whether Deligne's proof of Clozel's
theorem (Deligne 2009) can be transferred to motives. It may yield a theorem
with different hypotheses.
\end{aside}

\section{Examples}

By a motive in this section, we mean an object of $\Mot_{\mathrm{rat}%
}(\mathbb{F}{}_{q})$.

\subsection{Motives with semisimple Frobenius}

We say that a motive $X$ has semisimple Frobenius if the action of $\pi_{X}$
on $\omega(X)$ is semisimple. For example, the motives of curves, abelian
varieties, and $K3$ surfaces have semisimple Frobenius (Weil, Deligne,
Piatetski-Shapiro and Shafarevich). The property is preserved by passage to
direct sums, direct summands, tensor products, and duals, and so the motives
with semisimple Frobenius form a pseudo-abelian rigid tensor subcategory of
$\Mot_{\mathrm{rat}}(\mathbb{F}{}_{q})$. For such a motive $X$, $\mathbb{Q}%
{}[\pi_{X}]$ satisfies the conditions in \ref{xx12}, and we let $S$ denote the
associated algebraic group. We say that $\pi_{X}$ is \emph{regular} if $S$ is
generated by $\pi_{X}/q^{m/2}$.

\begin{theorem}
\label{xx32}Let $X$ be motive of weight $m$ over $\mathbb{F}{}_{q}$ with
semisimple Frobenius.

\begin{enumerate}
\item The $\mathbb{Q}{}$-algebra\textrm{ }$(\bigotimes\omega(X))^{S}$ is
generated by $\omega(X^{\otimes2})^{S}$.

\item If $X^{\otimes2}(m)$ satisfies the Tate conjecture and $\pi_{X}$ is
regular, then the full Tate conjecture holds for all motives in $\langle
X\rangle^{\otimes}$.

\item If $X^{\otimes2}(m)$ satisfies the Tate conjecture, then, for infinitely
many prime numbers $l$, the functor $\langle X_{l}\rangle^{\otimes}%
\rightarrow\langle X_{\mathrm{num}}\rangle^{\otimes}$ is faithful.
\end{enumerate}
\end{theorem}

\begin{proof}
These statements are special cases of Theorems \ref{xx01} and \ref{xx02}.
\end{proof}

\begin{remark}
\label{xx33}Let $V$ be a variety of even dimension $d$ with semisimple
Frobenius. After possibly extending the base field, we may suppose that $1$ is
the only eigenvalue of $\pi_{X}$ on $H_{l}^{d}(X)$ equal to a root of $1$.
Then we have a canonical decomposition%
\[
H_{l}^{d}(X)=H_{l}^{d}(X)^{\pi_{X}}\oplus H_{l}^{d}(X)_{\mathrm{trans}}%
\]
invariant under $\pi_{X}$ (trans\thinspace=\thinspace transcendental). If
$\pi_{X}$ acts on $H_{l}^{d}(X)_{\mathrm{trans}}$ with distinct eigenvalues,
then $H_{l}^{2d}(V\times V)^{\pi_{V}}$ is generated by $H_{l}^{d}(V)^{\pi_{X}%
}\otimes H_{l}^{d}(V)^{\pi_{X}}$ and the cohomology class of the graph of
$\pi_{X}$ (\cite{zarhin1996}, 4.4). It follows that the Tate conjecture holds
for $V\times V$ if it holds for $V$.
\end{remark}

\subsection{Abelian motives}

A motive in $\Mot_{\mathrm{rat}}(\mathbb{F}{}_{q})$ is said to be abelian (or
of abelian type) if it isomorphic to a motive $(A,e,m)$ with $A$ an abelian
variety. The category of abelian motives is the smallest pseudo-abelian rigid
tensor subcategory of $\Mot_{\mathrm{rat}}(\mathbb{F}{}_{q})$ containing the
motives of curves.

The standard conjecture of Lefschetz type holds for abelian motives, and so an
abelian motive $X$ and its dual $X^{\vee}$ are isomorphic. Thus, the full Tate
conjecture holds for $X$ if $T(X,l)$ and $S(X,l)$ hold for some $l$ (see
\ref{xx06}).

\begin{theorem}
\label{xx34}Let $X$ be an abelian motive of weight $m$ over $\mathbb{F}{}_{q}$
such that $X^{\otimes2}(m)$ satisfies the Tate conjecture (e.g., $X=h^{1}A$
for $A$ an abelian variety).

\begin{enumerate}
\item If $\pi_{X}$ is regular, then the full Tate conjecture holds for all
motives in $\langle X\rangle^{\otimes}.$

\item For infinitely many prime numbers $l$, the functor $\langle X_{l}%
\rangle^{\otimes}\rightarrow\langle X_{\mathrm{num}}\rangle^{\otimes}$ is faithful.
\end{enumerate}
\end{theorem}

\begin{proof}
These statements are special cases of Theorems \ref{xx01} and \ref{xx02}.
\end{proof}

\subsection{$K3$ surfaces}

Recall that the Tate conjecture is known for $K3$ surfaces over finite fields
(Artin and Swinnerton Dyer 1973 for elliptic $K3$ surfaces; Nygaard 1983 for
$K3$ surfaces of height $1;$Nygaard and Ogus 1985 for $K3$ of finite height
and $p\neq2,3$; Charles 2013 and Maulik 2014 independently for supersingular
$K3$ surfaces and $p\neq2,3$; Madapusi Pera 2015 for all $K3$ surfaces and
$p\neq2$; Madupusi Pera and Kim 2018 for all $K3$ surfaces and $p=2$). More
recently, the Tate conjecture has been proved for the square $V\times V$ of a
$K3$ surface $V$ (\cite{ito2018}).

\begin{theorem}
\label{xx35}Let $V$ be a $K3$ surface over a finite field.

\begin{enumerate}
\item If $\pi_{V}$ is regular, then the full Tate conjecture holds for all
powers of $V$.

\item There exists an infinite set of prime numbers $l$ such that
$l$-homological equivalence coincides with numerical equivalence for all
powers of $V$.
\end{enumerate}
\end{theorem}

\begin{proof}
These statements are special cases of Theorems \ref{xx01} and \ref{xx02}.
\end{proof}

It follows from Example \ref{xx15} that $\pi_{V}$ is regular if $V$ has height
$1$.

\bibliographystyle{cbe}
\bibliography{refTFF}

\begin{thebibliography}{}

\bibitem[\protect\astroncite{Albert}{1939}]{albert1939}
{\sc Albert, A.~A.} 1939.
\newblock Structure of {A}lgebras.
\newblock American Mathematical Society Colloquium Publications, vol. 24.
  American Mathematical Society, New York.

\bibitem[\protect\astroncite{Clozel}{1999}]{clozel1999}
{\sc Clozel, L.} 1999.
\newblock Equivalence num\'erique et \'equivalence cohomologique pour les
  vari\'et\'es ab\'eliennes sur les corps finis.
\newblock {\em Ann. of Math. (2)} 150:151--163.

\bibitem[\protect\astroncite{Deligne}{2009}]{deligne2009}
{\sc Deligne, P.} 2009.
\newblock Equivalence num\'erique, \'equivalence cohomologique, et th\'eorie de
  {L}efschetz des vari\'et\'es ab\'eliennes sur les corps finis (written by
  {L}. {C}lozel).
\newblock {\em Pure Appl. Math. Q.} 5:495--506.

\bibitem[\protect\astroncite{Deligne and Milne}{1982}]{deligneM1982}
{\sc Deligne, P. and Milne, J.~S.} 1982.
\newblock Tannakian categories, pp. 101--228.
\newblock {\em In} Hodge cycles, motives, and {S}himura varieties, Lecture
  Notes in Mathematics 900. Springer-Verlag, Berlin.

\bibitem[\protect\astroncite{Ito et~al.}{2018}]{ito2018}
{\sc Ito, K., Ito, T., and Koshikawa, T.} 2018.
\newblock {CM} liftings of {K}3 surfaces over finite fields and their
  applications to the {T}ate conjecture.
\newblock arXiv:1809.09604.

\bibitem[\protect\astroncite{Kowalski}{2006}]{kowalski2006}
{\sc Kowalski, E.} 2006.
\newblock Weil numbers generated by other {W}eil numbers and torsion fields of
  abelian varieties.
\newblock {\em J. London Math. Soc. (2)} 74:273--288.

\bibitem[\protect\astroncite{Lenstra and Zarhin}{1993}]{lenstraZ1993}
{\sc Lenstra, Jr., H.~W. and Zarhin, Y.~G.} 1993.
\newblock The {T}ate conjecture for almost ordinary abelian varieties over
  finite fields, pp. 179--194.
\newblock {\em In} Advances in number theory (Kingston, ON, 1991), Oxford Sci.
  Publ. Oxford Univ. Press, New York.

\bibitem[\protect\astroncite{Milne}{1999}]{milne1999}
{\sc Milne, J.~S.} 1999.
\newblock Lefschetz classes on abelian varieties.
\newblock {\em Duke Math. J.} 96:639--675.

\bibitem[\protect\astroncite{Saavedra~Rivano}{1972}]{saavedra1972}
{\sc Saavedra~Rivano, N.} 1972.
\newblock Cat\'egories {T}annakiennes.
\newblock Lecture Notes in Mathematics, Vol. 265. Springer-Verlag, Berlin.

\bibitem[\protect\astroncite{Scholl}{1994}]{scholl1994}
{\sc Scholl, A.~J.} 1994.
\newblock Classical motives, pp. 163--187.
\newblock {\em In} Motives (Seattle, WA, 1991), volume~55 of {\em Proc. Sympos.
  Pure Math.} Amer. Math. Soc.

\bibitem[\protect\astroncite{Tate}{1994}]{tate1994}
{\sc Tate, J.~T.} 1994.
\newblock Conjectures on algebraic cycles in {$l$}-adic cohomology, pp. 71--83.
\newblock {\em In} Motives (Seattle, WA, 1991), volume~55 of {\em Proc. Sympos.
  Pure Math.} Amer. Math. Soc., Providence, RI.

\bibitem[\protect\astroncite{Zarhin}{1991}]{zarhin1991}
{\sc Zarhin, Y.~G.} 1991.
\newblock Abelian varieties of {$K3$} type and {$l$}-adic representations, pp.
  231--255.
\newblock {\em In} Algebraic geometry and analytic geometry (Tokyo, 1990),
  ICM-90 Satell. Conf. Proc. Springer, Tokyo.

\bibitem[\protect\astroncite{Zarhin}{1996}]{zarhin1996}
{\sc Zarhin, Y.~G.} 1996.
\newblock The {T}ate conjecture for powers of ordinary {$K3$} surfaces over
  finite fields.
\newblock {\em J. Algebraic Geom.} 5:151--172.

\end{thebibliography}

\end{document}